\documentclass[12pt]{amsart}
\usepackage[utf8]{inputenc}
\usepackage[margin=1in]{geometry}

\usepackage[
    backend=biber,
    style=ieee,
    citestyle=numeric-comp,
    sorting=nyt,
    dashed=false
  ]{biblatex}
\usepackage{hyperref}
\usepackage[hyphenbreaks]{breakurl}
\usepackage{xurl}
\hypersetup{
    colorlinks=true,
    linkcolor=blue,
    citecolor=red,
    breaklinks=true
}

\usepackage{graphicx}
\usepackage{amsmath}
\usepackage{amssymb}
\usepackage{appendix}
\usepackage{amsthm}
\usepackage{booktabs}
\usepackage{graphicx}
\usepackage{subcaption}
\usepackage{comment}
\usepackage{enumerate}
\usepackage{esint}

\usepackage[dvipsnames]{xcolor}

\newtheorem{theorem}{Theorem}[section]
\newtheorem{corollary}{Corollary}[theorem]
\newtheorem{prop}[theorem]{Proposition}
\newtheorem{lemma}[theorem]{Lemma}

\theoremstyle{definition}

\theoremstyle{remark}
\newtheorem*{remark}{Remark}

\usepackage{amsmath}
\usepackage{amssymb}
\usepackage{mathrsfs}

\newcommand{\R}{{\mathbb R}}
\newcommand{\g}{{\gamma}}
\newcommand{\E}{{\mathcal{E}}}
\newcommand{\F}{{\mathcal{F}}}

\newcommand{\G}{\Gamma}

\usepackage{todonotes}


\title[Uniform Restriction for Simple Curves of Bounded Frequency]{Uniform Fourier Restriction Estimate for Simple Curves of Bounded Frequency}

\author{J. de Dios Pont}
\address{Department of Mathematics, University of California, Los Angeles, CA}
\email{jdedios@math.ucla.edu}

\author{H. J. Samuelsen}
\address{Department of Mathematical Sciences,
         Norwegian University of Science and Technology,
         Trondheim, Norway}

\email{helge.j.samuelsen@ntnu.no}
\date{\today}

\begin{document}

\begin{abstract}
    In this paper we prove a uniform Fourier restriction estimate over the class of simple curves where the last coordinate function can be extended to a holomorphic function of bounded frequency in a sufficiently large disc. The proof is based on a decomposition scheme for this class of functions.
\end{abstract}
\maketitle

\section{Introduction and main result}

The Fourier restriction operator $\mathcal{R}_\g$ associated to a curve $\g:I\to\R^d$ is defined as
\[
\mathcal{R}_\g(f)(t):=\widehat{f}(\g(t)),
\]
where $f\in\mathscr{S}(\R^d)$ and $t\in I$. In this paper we investigate uniform boundedness of the restriction operator $\mathcal{R}_\g:L^p(\mathbb{R}^d)\to L^q(I;\lambda_\g dt)$, where the latter is equipped with the affine arc length measure $\lambda_\g dt=|L_\g|^{\frac{2}{d^2+d}}dt$ and $L_\g=\det{(\g',\ldots,\g^{(d)})}$ denotes the torsion of the curve. The curves we consider are called \textit{simple} curves, and are of the form
\begin{equation}\label{defSimpCurve}
\g(t)=\left(t,t^2,\ldots, t^{d-1},\varphi(t)\right),
\end{equation}
where $\varphi$ belongs to certain class of functions, the so-called functions of bounded frequency.

The study of Fourier restriction of curves in dimension two can be traced back to Fefferman and Stein \cite{Fefferman}, Zygmund \cite{Zygmund}, and H\"{o}rmander \cite{Hormander}. The first uniform restriction estimate for curves is due to Sj\"{o}lin \cite{Sjolin}. He proved the existence of a restriction constant uniform over all $C^2$ convex plane curves for the range $1\leq p< 4/3$ and $3q\leq p'$, and showed that this range is optimal by considering the curve $\g(t)=(t,exp(-t^{-1})\sin(t^{-k}))$ for $k>q$. This curve acts as a counter-example to the boundedness of $\mathcal{R}_\g$ due to the rapid oscillation of the curve near $t=0$. The result of Sj\"{o}lin has been extended by Fraccaroli, who proved uniformity over all convex plane curves \cite{Fraccaroli}.

The first result in dimension three belongs to Prestini, who proved restriction in a reduced range \cite{Prestini}. The work of Prestini was extended to higher dimensions by Christ in \cite{Christ} for the range
\begin{equation}\label{ChristRange}
q\leq \frac{2}{d^2+d}p',\quad 1\leq p<\frac{d^2+2d}{d^2+2d-2},
\end{equation}
now known as the Christ range. The maximal conjectured range of $p$ and $q$ is
\begin{equation}\label{FullRange}
  1\leq p<\frac{d^2+d+2}{d^2+d},\quad 1\leq q\leq \frac{2}{d^2+d}p^{\prime}.
\end{equation}
where $p^{\prime}=p/(p-1)$ denotes the H\"{o}lder conjugate of $p$. The optimality of the range of $p$ follows from an argument by Arkhipov, Chubarikov and Karatsuba \cite{Arkhipov}, while a scaling argument can be used to derive the endpoint case for the range of $q$. 

The first result beyond the Christ range is due to Drury, who proved the restriction estimate in the full range \eqref{FullRange} for the moment curve in dimension three \cite{Drury}. 
He also pointed out the benefits of using the affine arc length measure rather than the Euclidean in the case of degenerate curves \cite{Drury-Affine}.

A well studied class of curves in uniform restriction is that of polynomial curves, meaning curves of the form $\g(t)=(p_1(t),\ldots,p_d(t))$ where $p_i$ are polynomials. 
A significant result is due to Dendrinos and Wright \cite{DW}. They proved restriction for the Christ range \eqref{ChristRange},
where the restriction constant only depends on $p,q,d$ and the maximal degree $N$. Stovall later extended the result to the full conjectured range \eqref{FullRange}. We emphasise that her result holds for any polynomial curve defined on $I=\mathbb{R}$ \cite{Stovall}.

In \cite{Chen}, Chen, Fan and Wang investigated simple curves of the form \eqref{defSimpCurve} for smooth functions $\varphi\in C^\infty(I)$ on a compact interval $I$. They proved a restriction estimate in the full range \eqref{FullRange}, under the additional condition that the torsion is bounded between two constants. The result is uniform in the sense that the constant only depends on $p$ and $d$. By considering a dyadic decomposition of the torsion, they extended the restriction estimate for the range
\[
1\leq p<\frac{d^2+d+2}{d^2+d},\quad 1\leq q< \frac{2}{d^2+d}p^{\prime}.
\]
However, the constant in the restriction now depends on a certain H\"{older} norm of the curve, and therefore is no longer uniform. Moreover, by building on the counter-example of Sj\"{o}lin, they were able to show that the estimate fails for $(d^2+d)q=2p'$ by considering the curve
\[
\g(t)=(t,\ldots,t^{d-1},e^{-t^{-\alpha}}\sin{t^{-\beta}}),
\]
for $\alpha>0$ and $2\beta>(d+1)\alpha$. This implies that a restriction estimate cannot hold for smooth simple curves at the critical line
$(d^2+d)q=2p'$. The result of Chen, Fan and Wang was extended to general smooth curves by Jesurum \cite{Jesurum}.

A class of curves where a restriction estimate could hold at the critical line ($(d^2+d)q=2p'$) is that of real analytic curves. By taking successive polynomial approximations of Sj\"{o}lin's example it is clear that even global analyticity does not suffice to obtain uniform bounds. Bounding some form of global oscillation of a family of curves becomes necessary.
In this article we consider a special class of real analytic simple curves, and prove a uniform restriction estimate for the full range \eqref{FullRange}.
Our class consists of real analytic curves of the form \eqref{defSimpCurve} where $\varphi$ has so-called bounded frequency. The frequency function goes back to the work of Agmon \cite{Agmon} and Almgren \cite{Almgren}, and has since become an important tool in the theory of elliptic PDEs. For a non-vanishing holomorphic function $u=\sum_{n=0}^\infty c_nz^n$ on $\Omega\subseteq \mathbb{C}$, we define the frequency function on $D(x,r)\subset \Omega$ as
\[
N_u(z_0,R)=2\frac{\sum_{n=1}^\infty n|c_n|^2R^{2n}}{\sum_{k=0}^\infty |c_k|^2R^{2k}}.
\]
If a function $u$ has frequency bounded by $N$ in a disc $D(x,r)$, then it behaves similar to a polynomial of degree $CN$ in every smaller disc centered at $x$.

Given a compact interval $I=[a-r,a+r]$, an integer $N\in\mathbb{N}$, and a positive constant $R>r$, we consider the classes of functions
\begin{align*}
    &\mathcal{A}(I,N,R):=\{\varphi: I\to \mathbb{R}: \exists \Phi\in\mathrm{Hol}(D(a,R))\text{ such that }\Phi|_I=\varphi,\text{ and } N_\Phi(a,R)\leq N\},\\
    &\mathcal{A}^{d-1}_0(I,N,R):=\{\varphi\in \mathcal{A}(I,N,R): \varphi(a)=\varphi'(a)=\ldots=\varphi^{(d-1)}(a)=0\}.
\end{align*}
Any simple curve \eqref{defSimpCurve} with $\varphi\in \mathcal{A}_0^{d-1}(I,N,R)$ is a smooth simple curve, so all restriction estimates in \cite{Chen} are valid for these curves. 
Our main result is the following.
\begin{theorem}\label{MainRestrictionTheorem}
Let $a\in\mathbb{R}$, and $r>0$ be fixed, and let $I$ denote the  bounded interval $I=[a-r,a+r]$. Then for any $N\in\mathbb{N}$, and any
\[
1\leq p< \frac{d^2+d+2}{d^2+d}
\]
there exists $C=C(N,d,p)>0$ such that for any simple curve
\[
\g(t)=\left(t,\ldots,t^{d-1},\varphi(t)\right),
\]
with $\varphi\in \mathcal{A}_0^{d-1}(I,N,2^{2d+3}r)$ and any $f\in L^p(\R^d)$ the restriction estimate
\begin{equation}
    \|\mathcal{R}_\g(f)\|_{L^{q}(I;\lambda_\g dt)}\leq C\|f\|_{L^p(\R^d)},
\end{equation}
holds for $q=2p'/(d^2+d)$.
\end{theorem}

We will use a decomposition scheme to prove Theorem \ref{MainRestrictionTheorem}. This was first done in \cite{DW} for polynomial curves. Here they showed that restriction can be reduced to having two properties of the curve;
\begin{enumerate}[i)]
    \item $\G(t_1,\ldots,t_d)=\sum_{i=1}^d\g(t_i)$ is $d!$-to-1.
    \item $|J_{\G}(t_1,\ldots,t_d)|\geq C\prod_{j=1}^d|L_\g(t_j)|^\frac1d\prod_{k>l}|t_k-t_l|$.
\end{enumerate}
Inequality $\mathrm{ii)}$ is known as \textit{the geometric inequality}. Uniform restriction estimates are then established by constructing a finite decomposition for polynomials curves on which these two properties hold. The same decomposition approach has successfully been used by de Dios Pont in \cite{DeDios1} to extend the uniform estimate of Stovall to complex polynomial curves. A further extension was later given by the same author in \cite{DeDios2}, where the decomposition was extended to include any algebraic extension of a $p$-adic field.

For our case, the torsion simplifies to $L_\g(t)=C_d \varphi^{(d)}(t)$, where $C_d$ is a positive constant only dependent on $d$. To prove Theorem \ref{MainRestrictionTheorem}, we will need the following decomposition theorem.
\begin{theorem}\label{MainDecompTheorem}
Let $a\in\mathbb{R}$, and $r>0$ be fixed, and let $I$ denote the  bounded interval $I=[a-r,a+r]$. For a fixed integer $N\in\mathbb{N}$, consider the function $\varphi\in\mathcal{A}_0^{d-1}(I,N,2^{d+3}r)$. There exists a finite decomposition,
\[
I=\bigcup_{j=1}^{M_{N}}I_j,
\]
into almost disjoint intervals such that $\varphi$ is single-signed on each $I_j$, there exist positive constants $c_j, C_j, K_j>0$, centres $a_j\in\mathbb{R}\backslash I_j$, and integers $k_j\leq 2N$ such that for each $t\in I_j$, 
\begin{equation}\label{PolyDecompEstimate}
    c_j|a_j-t|^{k_j}\leq |\varphi^{(d)}(t)|\leq C_j|a_j-t|^{k_j},
\end{equation}
and for each $(t_1,\ldots,t_d)\in I_j^d$,
\begin{equation}\label{GeometricInequality}
|J_{\Gamma}(t_1,\ldots,t_d)|\geq K_j\prod_{l=1}^d|\varphi^{(d)}(t_l)|^\frac{1}{d}\prod_{k>l}|t_k-t_l|.
\end{equation}
\end{theorem}
A consequence of \eqref{PolyDecompEstimate} and \eqref{GeometricInequality} is that $\G(t)=\sum\g(t_i)$ is injective on $I_j^d$ whenever $t_{\sigma(1)}<\ldots <t_{\sigma(d)}$ for any permutation $\sigma$ on $\{1,\ldots,d\}$.
There is nothing in the reduction argument of Dendrinos and Wright which requires the curve to be a polynomial curve. In fact, given Theorem \ref{MainDecompTheorem} the reduction argument provided in Section 3 of \cite{DW} can be used to achieve a uniform restriction estimate over the class $\mathcal{A}_0^{(d-1)}(I,N,R)$ in the Christ range \eqref{ChristRange}.
Finally, the result can be extend to the full range \eqref{FullRange} by following the ideas of Stovall from \cite{Stovall}.

\subsection{Structure of the paper}

In Section \ref{Preliminaries} we give a brief discussion of the frequency function, focusing on properties relevant to the proof of Theorem \ref{MainDecompTheorem}. For a more detailed description we refer to \cite{Malinnikova}.

The proof of Theorem \ref{MainDecompTheorem} is divided into two parts. We first find a finite decomposition of $I$ such that \eqref{PolyDecompEstimate} holds on each subinterval $I_j$ in Section \ref{Decomposition}. The main idea here is to write any function in the class $\mathcal{A}(I,N,R)$ as the product of a polynomial of bounded degree and a holomorphic factor. The latter cannot oscillate too much in a small neighbourhood, and thus the original function is locally equivalent to a polynomial of bounded degree. 
In Section \ref{GeometricInequalitySection} we finalize the proof of Theorem \ref{MainDecompTheorem} by deducing the geometric inequality \eqref{GeometricInequality}. 
The proof is based on a recursive integral formula for the Jacobian involving the torsion and its minors, first provided in \cite{DW}. Through an induction argument we show that property \eqref{PolyDecompEstimate} implies \eqref{GeometricInequality} for any $C^d$ simple curve.

Finally, in Section \ref{Reduction} we
prove Theorem \ref{MainRestrictionTheorem} by combining Lemma $1$ in \cite{Chen} with the argument of Stovall found in Section $3$ and $4$ of \cite{Stovall}.
It suffices to prove Theorem \ref{MainRestrictionTheorem} for the range
\[
p\geq \frac{d^2+d}{d^2+d-1},
\]
as the Christ range follows directly from Section 3 of \cite{DW}.

\section*{Acknowledgement}
The authors would like to thank Stanford University for hosting and supporting the visit that led to this collaboration in the fall quarter of 2022. A special thanks to Professor Eugenia Malinnikova for suggesting the problem and for excellent support. Partial funding was received from the Trond Mohn Foundation project Pure Mathematics in Norway and the UCLA Dissertation Fellowship Award.
\section{Properties of the Frequency function}\label{Preliminaries}
The frequency function has been a useful tool in the theory of elliptic partial differential operators. If we consider a holomorphic function $u$ on a disc $D(z_0,R)$, the frequency function reduces to
\[
N_u(z_0,R)=2\frac{\sum_{n=1}^\infty n|c_n|^2R^{2n}}{\sum_{k=0}^\infty |c_k|^2R^{2k}},
\]
where $c_n$ is the $n^{\mathrm{th}}$ coefficient of the power series expansion of $u$ around $z_0$. If there exists $N>0$ such that $N_u(z_0,R)\leq N$, we say that $u$ is a holomorphic function of bounded frequency at the point $z_0$. For a fixed point $z_0\in\mathbb{C}$, the frequency function is a monotonically increasing function of $r$ as 
\[
\frac{\partial N_u(z_0,r)}{\partial r}>0.
\]

One of the key properties of a holomorphic function of bounded frequency is that it behaves similarly to a polynomial of bounded degree. A straightforward estimate gives that the frequency function of any polynomial $p$ of degree at most $N$ is bounded by $2N$. That is $N_p(z_0,R)\leq 2N$ for any centre $z_0\in\mathbb{C}$ and any radius $R>0$.
The following proposition lists properties of holomorphic functions of bounded frequency.
\begin{prop}\label{FreqPropProp}
Let $N\in\mathbb{N}$ and let $u$ be a holomorphic function on $D(z_0,R)$ such that $N_u(z_0,R)\leq N$. Then the following holds.
\begin{enumerate}[i)]
    \item Given $z\in D(z_0,R)$, there exists $C>0$ such that for any $0<r<R/2$ with $D(z,r)\subset D(z_0,R/2)$, we have
    \begin{equation}\label{FrequencyDoublingLemma}
        \sup_{\zeta\in D(z,r)}|u(\zeta)|\leq 2^{CN+1}\sup_{\zeta\in D\left(z,\frac{r}{2}\right)}|u(\zeta)|.
    \end{equation}
    \item If $u(0)=0$, then there exists a constant $C>0$ such that the frequency function of the derivative is bounded in a smaller disc, that is \begin{equation}\label{BoundedFreqDerivative}
    N_{u'}\left(z_0,\frac{R}{2}\right)\leq C(N+1).
    \end{equation}
    \item There exists a universal constant $1/2<c<3/4$ such that
    \begin{equation}\label{NumberOfZeros}
        \#\left\{z\in D\left(z_0,cR\right): u(z)=0\right\}\leq 2N.
    \end{equation}
    In particular $\#\left\{z\in D\left(z_0,R/2\right): u(z)=0\right\}\leq 2N$.
\end{enumerate}
\end{prop}
We omit the proof, and instead refer to \cite{Malinnikova} and \cite{Han} for \eqref{NumberOfZeros}.

From \eqref{NumberOfZeros} it follows that one can decompose any function of bounded frequency into the product of a polynomial of bounded degree and a non-vanishing holomorphic factor. Moreover, the non-vanishing holomorphic factor is also of bounded frequency in a smaller disc. We refer to the following lemma by Foster \cite{Foster}.
\begin{lemma}[Lemma $9$ in \cite{Foster}]\label{lemmaByFoster}
Assume that $N_u(z_0,R)\leq N$, and that $u=pf$, where $p$ is a polynomial and $f$ is a non-vanishing holomorphic function on $D(z_0,R/2)$. Then there exists a universal constant $C>0$ such that $N_f(z_0,R/2)\leq CN$.
\end{lemma}

\section{Decomposition}\label{Decomposition}
In this section we give a proof of Theorem \ref{MainDecompTheorem}, but without the geometric inequality \eqref{GeometricInequality}. The proof consists of two parts. The first part is showing that for any non-vanishing holomorphic function $\psi$ with bounded frequency in a disc at least four times the radius of the original interval $I$, there exists a decomposition of $I$ into almost disjoint intervals $I_k$ on which we can control the oscillation of $\psi$. In particular, using Lemma \ref{lemmaByFoster} we can show that any holomorphic function with bounded frequency in an even larger disc must be equivalent to a polynomial of bounded degree on sufficiently small subintervals of $I$.

The second part of the proof is showing that for any function $\varphi\in\mathcal{A}(I,N,R)$, there exists a decomposition of $I$ such that $\varphi$ is equivalent to a monomial on each subinterval. Here we refer to a decomposition given in \cite{DW} for polynomials. Taking the intersection of the intervals from these two decomposition schemes gives the required decomposition.
Let us therefore begin with the decomposition for non-vanshing holomorphic functions with bounded frequency.
\begin{prop}\label{HolDecomProp}
Let $I=[a-r,a+r]$ be a compact interval, and let $\psi\in \mathcal{A}(I,N,4r)$ be non-vanishing on $D(a,4r)\subset \mathbb{C}$. Then for each $\varepsilon>0$ there exists a finite decomposition $I=\bigcup_{j=1}^{M_{N,\varepsilon}} I_j$ such that for each $j$
\begin{equation*}
    \left|\frac{\psi(x)}{\psi(y)}-1\right|<\varepsilon,
\end{equation*}
for any $x,y\in I_j$. The number of intervals $M_{N,\varepsilon}$ depends only on $N$ and $\varepsilon$.
\end{prop}
\begin{proof}
Without loss of generality, we may assume that $\sup_{z\in D(a,4r)}|\psi(z)|=1$. Since $\psi$ is of bounded frequency, it follows by \eqref{FrequencyDoublingLemma} in Proposition \ref{FreqPropProp} that
\[
\sup_{z\in D(x,2s)}|\psi(z)|\leq 2^{CN+1}\sup_{z\in D(x,s)}|\psi(z)|,
\]
for any $x\in I$, and $s<r$. In particular, for any $k\in\mathbb{N}$,
\[
1=\sup_{z\in D(a,4r)}|\psi(z)|\leq 2^{(k+2)(CN+1)}\sup_{z\in D(a,2^{-k}r)}|\psi(z)|.
\]
Fix some $k\in\mathbb{N}$, and divide the interval $I$ into $2^{k+1}$ subintervals $I_j$ of length $2^{-k}r$ with centres $a_j$. Then for each $j$ it follows that
\begin{equation}\label{BallEstimate}
    \begin{aligned}
   \sup_{z\in D(a_j,2^{-k}r)}|\psi(z)|\geq 2^{-(k+1)(CN+1)}\sup_{z\in D(a_j,2r)}|\psi(z)|&\geq 2^{-(k+1)(CN+1)}\sup_{z\in D(a,r)}|\psi(z)|
   \\&\geq 2^{-(k+3)(CN+1)}. 
\end{aligned}
\end{equation}

Define the function $
h(z)=-\log_2|\psi(z)|$ on $D(a,4r)$. This is a positive harmonic function as $\psi$ is a normalized non-vanishing holomorphic function on $D(a,4r)$. Moreover, by \eqref{BallEstimate},
\[
\inf_{z\in D(a_j,2^{-k}r)}h(z)\leq (k+3)(CN+1).
\]
By Harnack's inequality, we have
\begin{equation*} \frac{1-2^{-k}}{1+2^{-k}}h(a_j)\leq h(x)\leq \frac{1+2^{-k}}{1-2^{-k}}h(a_j),
\end{equation*}
for any $x\in D(a_j,2^{-k}r)\subset D(a_j,r)$. As such, we have
\[
\sup_{z\in D(a_j,2^{-k}r)}h(z)\leq \frac{(1+2^{-k})^2}{(1-2^{-k})^2}\inf_{z\in D(a_j,2^{-k}r)}h(z).
\]
Since $I_j\subset D(a_j,2^{-k}r)$, it follows that for any $x,y\in I_j$,
\begin{equation}\label{HarnackArgument}
\begin{aligned}
|h(x)-h(y)|
\leq& \sup_{z\in D(a_j,2^{-k}r)}h(z)-\inf_{z\in D(a_j,2^{-k}r)}h(z) 
\\
\leq& 2^{-k+4}\inf_{z\in D(a_j,2^{-k}r)}h(z)
\\
\leq& 2^{-k+4}(k+3)(CN+1).
\end{aligned}
\end{equation}
Thus, by choosing $k$ sufficiently large we can make the difference arbitrarily small.

Now let $\varepsilon>0$.
By the continuity of the exponential function there exists $\delta>0$, such that
\[
\left|\frac{\psi(x)}{\psi(y)}-1\right|=\left|2^{h(y)-h(x)}-1\right|<\varepsilon,
\]
whenever $|h(x)-h(y)|<\delta$. We see from \eqref{HarnackArgument} that this is achieved by Harnack's inequality by choosing $k=k(\varepsilon,N)\in\mathbb{N}$ such that
\[
2^{-k}(k+3)<\frac{\delta}{16(CN+1)},
\]
and decomposing $I$ into $2^{k+1}$ intervals of length $2^{-k}r$.
\end{proof}
\begin{remark}
From the proof it follows that the number of intervals $I_j$ is independent of the radius $r$. The only dependence on the length of the interval $I$ is in the size of the disc for which the function needs to have bounded frequency.
\end{remark}
As mentioned earlier, by Lemma \ref{lemmaByFoster} we note that any $\varphi\in \mathcal{A}(I,N,8r)$ can be written as $\varphi=\psi p$, where $p$ is a polynomial of bounded degree and $\psi\in\mathcal{A}(I,CN,4r)$. Proposition \ref{HolDecomProp} then gives that $\varphi$ is equivalent to the polynomial $p$ on each $I_j$ in the decomposition. Since a decomposition scheme for polynomials already exists, Theorem \ref{MainDecompTheorem} follows by combining these two decompositions. In order to continue, the following lemma is needed. 
\begin{lemma}[D1 in \cite{DW}]\label{D1Lemma}
Let $Q:J\to\R$ be a real polynomial of degree at most $N$. Then the interval $J$ can be decomposed into a finite number of open disjoint intervals,
\[
J=\bigcup_{j=1}^{M_{N}}I_j,
\]
so that on each $I_j$ there exist constants $c_j,C_j>0$ for which 
\[
c_j\left(A_j|t-a_j|^{k_j}\right)\leq|Q(t)|\leq C_j\left(A_j|t-a_j|^{k_j}\right),
\]
for all $t\in I_j$. Here $0\leq k_j\leq N$, $A_j\neq 0$, and the centres $a_j$ are the real parts of the zeros of $Q$ which are not contained in the interior of $I_j$. The constants $c_j,C_j$ and $M_{N}$ depend only on $N$.
\end{lemma}
With this decomposition lemma for polynomials, we are ready to prove the main theorem of this section.
\begin{theorem}\label{FirstDecompThm}
Let $N\in\mathbb{N}$ be a fixed integer, $I=[a-r,a+r]$ a compact interval, and consider $\varphi\in \mathcal{A}(I,N,8r)$. There exists a finite decomposition 
\[
I=\bigcup_{j=1}^{K_N}I_j,
\]
so that for each $j$ there exist constants $c_j,C_j>0$, $a_j\in\R\backslash I_j$, and $0\leq k_j\leq 2N$, such that 
\[
c_j|t-a_j|^{k_j}\leq |\varphi(t)|\leq C_j|t-a_j|^{k_j},
\]
for all $t\in I_j$. Moreover, the number of intervals $K_N$ depends only on $N$.
\end{theorem}
\begin{proof}
By Lemma \ref{lemmaByFoster}, there exists a polynomial of degree at most $2N$, and a non-vanishing function of bounded frequency $\psi$ of absolute value at most $1$, such that
\[
\varphi(x)=p(x)\psi(x),
\]
on the disc $D(a,4r)$. Let $\varepsilon=1/2$, and apply Proposition \ref{HolDecomProp} to the function $\psi$. Then there exists a finite number of intervals such that
\[
I=\bigcup_{j=1}^{M_{N}}I_j,
\]
on which
\begin{equation*}
    \left|\frac{\psi(x)}{\psi(y)}-1\right|\leq \frac{1}{2},
\end{equation*}
for any $x,y\in I_j$.
Fix a point $x_j\in I_j$, and consider the function $\psi_j$,
\begin{equation*}
    \psi_j(t)=\frac{\psi(t)}{\psi(x_j)},
\end{equation*}
which is well-defined as $\psi$ is non-vanishing. Moreover, we have
\begin{equation*}
    \left|\psi_j(t)-1\right|=\left|\frac{\psi(t)}{\psi(x_j)}-1\right|\leq\frac{1}{2}.
\end{equation*}
This necessarily implies that $1/2\leq|\psi_j(t)|\leq 3/2$ on $I_j$.

For each $I_j$, consider the polynomial $p_j$ given by
\[
p_j(t)=\psi(x_j)p(t),
\]
so that $\varphi=\psi_jp_j$ on $I_j$. Apply Lemma \ref{D1Lemma} with respect to the polynomial $p_j$ to decompose $I_j$ further into $I_{j,i}$. Then on each $I_{j,i}$ there exist constants $c_i,C_i$ such that
\[
c_i\leq \frac{|p_j(t)|}{A_i|t-a_i|^{k_i}}\leq C_i,
\]
by Lemma \ref{D1Lemma}.
However, this implies
\begin{align*}
    \frac{c_i}{2}\leq\frac{|\varphi(t)|}{A_i|t-a_i|^{k_i}}\leq \frac{3}{2}C_i,
\end{align*}
as $\varphi=\psi_jp_j$. The proof follows by relabeling the intervals $I_{j,i}$.
\end{proof}

For each function $\varphi\in \mathcal{A}_0^{d-1}(I,N,2^{d+3}r)$ it follows that $\varphi^{(d)}\in\mathcal{A}(I,\mathcal{N}_{N,d},8r)$ by applying  \eqref{BoundedFreqDerivative} $d$ times. Here $\mathcal{N}_{N,d}$ is a positive constant depending only on $N$ and $d$. Moreover, by property \eqref{NumberOfZeros} we control the number of zeros of $\varphi^{(d)}$, and thus the number of connected components of $I$ on which $\varphi^{(d)}$ is single-signed. The first part of Theorem \ref{MainDecompTheorem} then follows from Theorem \ref{FirstDecompThm} applied to $\varphi^{(d)}$ and taking the intersection of all intervals in the decomposition with all connected components of $I$ on which $\varphi^{(d)}$ is single-signed.

\section{The Geometric Inequality}\label{GeometricInequalitySection}
In this section we will finalize the proof of Theorem \ref{MainDecompTheorem} by proving the geometric inequality \eqref{GeometricInequality}. We will use an induction argument based on a recursive integral formula for the Jacobian. This was first done in \cite{DW}, and we refer to Section $5$ in the same paper for a full derivation of the formula.
\begin{prop}\label{DW-Prop}
Given a curve $\gamma:I\to\R^d$, let $L_d=L_\g$ denote the torsion of the curve, and $L_i$ denote the determinant of the $i\times i$ minor matrix,
\[
L_i(t)=\det\begin{pmatrix}
\g_1'(t)&\ldots&\g_1^{(i)}(t)\\
\vdots&\ddots&\vdots\\
\g_i'(t)&\ldots&\g_i^{(i)}(t)
\end{pmatrix}.
\]
Define the map $\Gamma:I^d\to\R^d$ given by,
\[
\Gamma(t_1,\ldots,t_d)=\sum_{i=1}^d\g(t_i).
\]
Then the Jacobian of $\Gamma$ is given by
\[
J_\Gamma(t_1,\ldots,t_d)=\Lambda_d(t_1,\ldots,t_d),
\]
where $\Lambda_j$ satisfies the recursive formula,
\begin{align*}
    \Lambda_1(t)=&\frac{L_{d-2}(t)L_d(t)}{L_{d-1}(t)^2}\\
    \Lambda_j(t_1,\ldots,t_j)=&\prod_{i=1}^j\frac{L_{d-j-1}(t_i)L_{d-j+1}(t_i)}{L_{d-j}(t_i)^2}\int_{t_1}^{t_2}\ldots \int_{t_{j-1}}^{t_{j}}\Lambda_{j-1}(s_1,\ldots,s_{j-1})ds_{j-1}\ldots ds_1.
\end{align*}
Here the convention $L_0(t)=L_{-1}(t)\equiv 1$ is used.
\end{prop}
As mentioned earlier, for simple curves the torsion $L_\g(t)=C_d\varphi^{(d)}(t)$ has a rather simple expression. Moreover, for each $i\times i$ minor we have $L_i(t)=C_i>0$. This gives an immediate corollary of Proposition \ref{DW-Prop} which simplifies the recursive formula. It is this simplification which is exploited below in giving an independent proof of the geometric inequality for simple curves.
\begin{corollary}\label{DW-Corollary}
Let $\varphi\in C^d(I)$, and let $\g:I\to\R^d$ be a simple curve of the form
\[
\g(t)=\left(t,t^2,\ldots,t^{d-1},\varphi(t)\right).
\]
Then following the notation of Proposition \ref{DW-Prop},
\begin{equation*}
    J_\Gamma(t_1,\ldots,t_d)=\Lambda_d(t_1,\ldots,t_d),
\end{equation*}
where the $\Lambda_j$ satisfies the simplified recursive formula,
\begin{align*}
    \Lambda_1(t)=C_d\varphi^{(d)}(t),\quad
    \Lambda_j(t_1,\ldots,t_j)=C_{j,d}\int_{t_1}^{t_2}\ldots \int_{t_{j-1}}^{t_{j}}\Lambda_{j-1}(s_1,\ldots,s_{j-1})ds_{j-1}\ldots ds_1,
\end{align*}
where $C_{j,d}>0$ is a positive combinatorical constant only dependent on $j$ and $d$.
\end{corollary}

The following lemma will serve as a base case for the induction argument. It will also be used several times during the induction argument itself.
\begin{lemma}\label{SingleIntegralInequality}
Let $\alpha\in \mathbb{R}_+$ and $t<\tau$. Then for any $a\notin [t,\tau]$,
\begin{equation}\label{BaseCaseEq}
    \int_t^\tau|s-a|^\alpha ds\geq C_\alpha|t-a|^\frac{\alpha}{2}|\tau-a|^\frac{\alpha}{2}|\tau-t|.
\end{equation}
\begin{proof}
Since $a\notin [t,\tau]$, there exist integers $n,m\in\mathbb{Z}$ such that
\[
2^{n-1}\leq |\tau-a|<2^{n},\quad 2^{m-1}\leq |t-a|<2^{m}.
\]
If $a<t$, then $m\leq n$.
Assume first that $n\geq m+1$. Then
\begin{align*}
    \int_{t}^{\tau}|s-a|^{\alpha}ds
    \geq& |t-a|^{\frac{\alpha}{2}}\int_{t}^{\tau}|s-a|^{\frac{\alpha}{2}}ds\\
    =& \frac{2}{\alpha+2}|t-a|^{\frac{\alpha}{2}}\left(|\tau-a|^{\frac{\alpha}{2}+1}-|t-a|^{\frac{\alpha}{2}+1}\right)\\
    =& \frac{2}{\alpha+2}|t-a|^{\frac{\alpha}{2}}|\tau-a|^{\frac{\alpha}{2}+1}\left(1-\frac{|t-a|^{\frac{\alpha}{2}+1}}{|\tau-a|^{\frac{\alpha}{2}+1}}\right)\\
    \geq&
    \frac{2}{\alpha+2}|t-a|^{\frac{\alpha}{2}}|\tau-a|^{\frac{\alpha}{2}+1}\left(1-2^{\left(\frac{\alpha}{2}+1\right)(m-n)}\right)\\
    \geq&
    \frac{1}{\alpha+2}|t-a|^{\frac{\alpha}{2}}|\tau-a|^{\frac{\alpha}{2}}|\tau-t|,
\end{align*}
where we have used that $|\tau-a|\geq |\tau-t|$ since $a<t<\tau$.
This proves \eqref{BaseCaseEq} in the case $n\geq m+1$. If $n=m$, then
\begin{align*}
    \int_{t}^{\tau}|s-a|^{\alpha}ds
    \geq \frac{2^{n\frac{\alpha}{2}}2^{n\frac{\alpha}{2}}}{2^{\alpha}}(\tau-t)
    \geq \frac{1}{2^{\alpha}}|t-a|^\frac{\alpha}{2}|\tau-a|^\frac{\alpha}{2}|\tau-t|,
\end{align*}
which concludes the proof with
\begin{equation*}
    C_\alpha=\min\left\{\frac{1}{\alpha+2},\frac{1}{2^\alpha}\right\}.
\end{equation*}

If $a>\tau$, then $m\geq n$, and $|t-a|\geq |\tau-a|$. The proof then follows by a similar argument as for the case $a<t$.
\end{proof}
\end{lemma}
\begin{prop}\label{GeoIneq}
Let $\gamma:I\to\R^d$ be a simple curve of the form\[
\g(t)=\left(t,t^2,\ldots,t^{d-1},\varphi(t)\right),\]
for some $\varphi\in C^d(I)$.
Assume that there exists an interval $I_0\subset I$ such that $\varphi^{(d)}$ is single-signed, and that there exist constants $C_1,C_2>0$, $\kappa\in\mathbb{Z}_+$ and $a\in \mathbb{R}\backslash I_0$ such that for any $t\in I_0$ 
\[
C_1|t-a|^\kappa\leq |\varphi^{(d)}(t)|\leq C_2|t-a|^\kappa.
\]
Then for any $(t_1,\ldots,t_d)\in I_0^d$ 
\[
|J_\Gamma(t_1,\ldots,t_d)|\geq C\prod_{j=1}^d|\varphi^{(d)}(t_j)|^{\frac{1}{d}}\prod_{l<k}|t_k-t_l|.
\]
\end{prop}
\begin{proof}
Since rearranging $t_i$ in the Jacobian only affects the sign, we may without loss of generality assume $t_1\leq t_2\ldots\leq t_d$. Moreover, as $\varphi^{(d)}$ is singled-signed on $I_0$, we also have,
\[
|\Lambda_2(t_1,t_2)|=C_d\int_{t_1}^{t_2}|\varphi^{(d)}(s)|ds,
\]
and so we may safely assume that all $\Lambda_l$ are positive. 

By corollary \ref{DW-Corollary}, the proof is complete if we can show that
\begin{equation}\label{inductionStep}
\Lambda_l(t_1,\ldots,t_l)\geq C\prod_{j=1}^l|\varphi^{(d)}(t_j)|^\frac{1}{l}\prod_{1\leq i<k\leq l}|t_k-t_i|
\end{equation}
for all $l$. This will be verified with an induction argument.

We start with the base case $l=2$. This follows from Lemma \ref{SingleIntegralInequality}, since
\begin{align*}
    \Lambda_2(t_1,t_2)=C_d\int_{t_1}^{t_2}|\varphi^{(d)}(s)|ds
    \geq& C \int_{t_1}^{t_2}|s-a|^{\kappa}ds\\
    \geq&
    C|t_1-a|^{\frac{\kappa}{2}}|t_2-a|^{\frac{\kappa}{2}}|t_2-t_1|
    \\
    \geq&
    C|\varphi^{(d)}(t_1)|^{\frac{1}{2}}|\varphi^{(d)}(t_2)|^{\frac{1}{2}}|t_2-t_1|.
\end{align*}

Assume now that \eqref{inductionStep} holds for some $l-1$. Then
\begin{align}\label{LambaIntegral1}
    \Lambda_l(t_1,\ldots,t_l)
    =&C
    \int_{t_{l-1}}^{t_l}\ldots\int_{t_1}^{t_2}\Lambda_{l-1}(s_1,\ldots,s_{l-1})ds_1\ldots ds_{l-1}\nonumber\\ 
    \geq&C
    \int_{t_{l-1}}^{t_l}\ldots\int_{t_1}^{t_2}\prod_{j=1}^{l-1}|\varphi^{(d)}(t_j)|^\frac{1}{l-1}\prod_{1\leq i<k\leq l-1}|s_k-s_i|ds_1\ldots ds_{l-1}\nonumber\\
    \geq& C
    \int_{t_{l-1}''}^{t_l'}\ldots\int_{t_1''}^{t_2'}\prod_{j=1}^{l-1}|s_j-a|^\frac{\kappa}{l-1}\prod_{1\leq i<k\leq l-1}|s_k-s_i|ds_1\ldots ds_{l-1},
\end{align}
where the points $t_i'$ and $t_i''$ are defined in such a way that we only integrate over the middle third of each interval, that is
\begin{align*}
    t_i':=&\frac{2t_{i}+t_{i-1}}{3},\quad 2\leq i\leq l,  
    \\t_{i}'':=&\frac{t_{i+1}+2t_i}{3},\quad 1\leq i\leq l-1.
\end{align*}

For any choice of $1\leq i<j\leq l-1$, consider  $s_i\in[t_i'',t_{i+1}']$, and $s_j\in[t_j'',t_{j+1}']$. Since $j>i$, it then follows that
\[
s_j-s_i\geq t_{j}''-t_{i+1}'=\frac{t_{j+1}+2t_j}{3}-\frac{2t_{i+1}+t_{i}}{3}=\frac{t_{j+1}-t_i}{3}+2\frac{t_{j}-t_{i+1}}{3}
\geq \frac{t_{j+1}-t_i}{3},\]
where we used that $t_j\geq t_{i+1}$. This implies
\[
\prod_{1\leq i<j\leq l-1}|s_j-s_i|\geq 3^{-\frac{(l-1)(l-2)}{2}}\prod_{1\leq i<j\leq l-1}|t_{j+1}-t_i|,
\]
which is the part of the Vandermonde determinant for which the indices differ by at least two.
Inserted into \eqref{LambaIntegral1}, and applying Lemma \ref{SingleIntegralInequality} with $\alpha=\kappa/(l-1)$, we see that
\begin{align}\label{LambaIntegral2}
    \Lambda_l(t_1,\ldots,t_l)
    \geq&
    C\prod_{1\leq i<j\leq l-1}|t_{j+1}-t_i|\prod_{k=1}^{l-1}\int_{t''_k}^{t'_{k+1}}|s_k-a|^\frac{\kappa}{l-1}ds_k
    \nonumber\\
    \geq&
    C\prod_{1\leq i<j\leq l-1}|t_{j+1}-t_i|\prod_{k=1}^{l-1}|t_{k+1}'-a|^\frac{\kappa}{2(l-1)}|t_{k}''-a|^{\frac{\kappa}{2(l-1)}}|t'_{k+1}-t_{k}''|.
\end{align}
Note that for any $1\leq k\leq l-1$,
\[
t_{k+1}'-t_k''=\frac{1}{3}(2t_{k+1}+t_k-(t_{k+1}+2t_k))=\frac{t_{k+1}-t_k}{3},
\]
which reduces \eqref{LambaIntegral2} to
\begin{equation}\label{LambdaIntegral3}
    \Lambda_l(t_1,\ldots,t_l)
    \geq
    \frac{C}{3^{l-1}}\prod_{1\leq i<j\leq l-1}|t_{j}-t_i|\prod_{k=1}^{l-1}|t_{k+1}'-a|^\frac{\kappa}{2(l-1)}|t_{k}''-a|^{\frac{\kappa}{2(l-1)}}.
\end{equation}
Moreover, for fixed $k$ and any $0<\beta_k< 1$, we have
\begin{align*}
    (t'_{k+1}-a)^\frac{\kappa}{2(l-1)}=\left(\frac{(t_k-a)+2(t_{k+1}-a)}{3}\right)^\frac{\kappa}{2(l-1)}
\geq&
C_{\beta_k}\left(\beta_k (t_{k}-a)+(1-\beta_k)(t_{k+1}-a)\right)^\frac{\kappa}{2(l-1)}\\
\geq&
C_{\beta_k}(t_{k}-a)^{\beta_k\frac{\kappa}{2(l-1)}}(t_{k+1}-a)^{(1-\beta_k)\frac{\kappa}{2(l-1)}},
\end{align*}
by the arithmetic-geometric mean inequality. Here the constant $C_{\beta_k}$ is
\[
C_{\beta_k}=\min\left(\frac{3\beta_k}{2},3(1-\beta_k)\right)^\frac{\kappa}{2(l-1)}.
\]Similarly, we can write
\begin{equation}\label{tk-estimate}
\begin{aligned}
    (t''_{k}-a)^\frac{\kappa}{2(l-1)}=&\left(\frac{2(t_k-a)+(t_{k+1}-a)}{3}\right)^\frac{\kappa}{2(l-1)}\\
\geq&
C_{\beta_k}\left(\beta_k (t_{k}-a)+(1-\beta_k)(t_{k+1}-a)\right)^\frac{\kappa}{2(l-1)}\\
\geq&
C_{\beta_k}(t_{k}-a)^{\beta_k\frac{\kappa}{2(l-1)}}(t_{k+1}-a)^{(1-\beta_k)\frac{\kappa}{2(l-1)}}.
\end{aligned}
\end{equation}
Using \eqref{tk-estimate} to estimate the right hand side of \eqref{LambdaIntegral3} one obtains, for any values $0< \beta_i<1$,
\begin{align*}
    &\prod_{k=1}^{l-1}|t_{k+1}'-a|^\frac{\kappa}{2(l-1)}|t_{k}''-a|^{\frac{\kappa}{2(l-1)}}\\
    &\geq C |t_1-a|^{\frac{\kappa}{2(l-1)}2\beta_1}|t_l-a|^{\frac{\kappa}{2(l-1)}2(1-\beta_{l-1})}\prod_{i=2}^{l-1}|t_i-a|^{\frac{\kappa}{2(l-1)}2(1+\beta_i-\beta_{i-1})}.
\end{align*}
The proof would follow if we can choose each exponent equal to $\kappa/l$. This condition gives rise to a linear system of equations,
\begin{align*}
\beta_1=\frac{l-1}{l},\quad \beta_i=\frac{l-1}{l}-1+\beta_{i-1}=\beta_{i-1}-\frac{1}{l},\quad \beta_{l-1}=1-\frac{l-1}{l}=\frac{1}{l}.
\end{align*}
By recursion, the solution is given by 
\begin{equation}\label{BetaChoice}
    \beta_{i}=\frac{l-i}{l}\in (0,1),
\end{equation}
for each $1\leq i \leq l-1$.
Thus, choosing $\beta_i$ as in \eqref{BetaChoice} we conclude that
\[
\prod_{k=1}^{l-1}|t_{k+1}'-a|^\frac{\kappa}{2(l-1)}|t_{k}''-a|^{\frac{\kappa}{2(l-1)}}
\geq C\prod_{j=1}^l|t_j-a|^\frac{\kappa}{l},
\]
which inserted into \eqref{LambdaIntegral3} gives
\begin{equation*}
    \Lambda_l(t_1,\ldots,t_l)
    \geq
    \frac{C}{3^{l-1}}\prod_{1\leq i<j\leq l-1}|t_{j}-t_i|\prod_{j=k}^l|t_k-a|^\frac{\kappa}{l}\geq C\prod_{j=k}^l|\varphi^{(d)}(t_k)|^\frac{1}{l}\prod_{1\leq i<j\leq l-1}|t_{j}-t_i|.
\end{equation*}
This concludes the proof.
\end{proof}

\begin{remark}
Once the assumptions of Proposition \ref{GeoIneq} are met, the geometric inequality follows directly from the polynomial geometric inequality given in \cite{DW}. This has to do with the recursive formula. Let $\upsilon(t)=(t,\ldots,t^{d-1},p(t))$ denote the polynomial simple curve with $p^{(d)}(t)=(t-a)^\kappa$.
Moreover, let $\Lambda_j$ be as in Proposition \ref{GeoIneq}, and let $\Sigma_j$ denote the recursive integrals associated to the curve $\upsilon$. From the equivalence of $\varphi^{(d)}$ and $p^{(d)}$, it follows that
\begin{align*}
|J_\Gamma(t_1,\ldots,t_d)|=|\Lambda_d(t_1,\ldots,t_d)|\geq C|\Sigma_d(t_1,\ldots,t_d)|=C|J_{\Upsilon}(t_1,\ldots,t_d)|,
\end{align*}
where $\Upsilon(t)=\sum_{j=1}^d\upsilon(t_i)$. However, since there already exists a geometric inequality for polynomial curves, we know that
\[
|J_\Gamma(t_1,\ldots,t_d)|\geq C|J_{\Upsilon}(t_1,\ldots,t_d)|\geq C\prod_{i=1}^d|t_i-a|^\frac{\kappa}{d}\prod_{k>l}|t_i-t_k|
\geq C\prod_{i=1}^d|\varphi^{(d)}(t_i)|^\frac{1}{d}\prod_{k>l}|t_i-t_k|,
\]
which then proves the geometric inequality.
\end{remark}
\section{Fourier Restriction}\label{Reduction}
In this section we prove Theorem \ref{MainRestrictionTheorem} using the techniques of Stovall from Section 4 in \cite{Stovall}. As is customary in the theory of Fourier restriction, we will consider the dual formulation. For a curve $\g:I\to\mathbb{R}^d$ we define the weighed and unweighted extension operators $\E_\g$ and $\mathcal{F}_\g$ to be defined as
\[
\E_\g f(x)=\int_I f(t)e^{2\pi i x\cdot \g(t)}\lambda_\g dt,\quad \F_\g f(x)=\int_I f(t)e^{2\pi i x\cdot \g(t)}dt.
\]
Theorem \ref{MainRestrictionTheorem} is equivalent to showing that the operator norm of the weighed extension operator is bounded, 
\[
\|\E_\g f\|_{L^{p'}(\R^d)}\leq C\|f\|_{L^{q'}(I;\lambda_\g dt)},
\]
where $\lambda_\g=|L_\g(t)|^\frac{2}{d^2+d}$, and where $C$ is uniform over the class $\mathcal{A}_0^{d-1}(I,N,R)$. For the proof of Theorem \ref{MainRestrictionTheorem}, we need the following uniform restriction lemma for the unweighted extension operator.
\begin{lemma}[Lemma $1$ in \cite{Chen}]\label{Restriction-Chen}
Let $d\geq 3$ and let $\gamma$ be a simple curve with $\varphi\in C^d(I)$ which satisfies
\[
\frac{1}{2}\leq |\varphi^{(d)}(t)|\leq 1,
\]
for all $t\in I$. Then
\[
\|\F_\g f\|_{L^{p'}(\R^d)}\leq C\|f\|_{L^{q'}(I)},
\]
holds for all $p$ and $q$ in the range
\[
1\leq p< \frac{d^2+d+2}{d^2+d},\quad q=\frac{2}{d^2+d}p',
\]
and the constant $C$ depends only on $p$ and $d$.
\end{lemma}
There is an immediate corollary of Lemma \ref{Restriction-Chen} using the affine invariance of the weighted extension operator on the critical line due to the affine arc length measure.
\begin{corollary}\label{Restriction-Chen-Cor}
Let $d\geq 3$, $M>0$, and let $\gamma$ be a simple curve with $\varphi\in C^d(I)$ which satisfies
\begin{equation}\label{assumptionsResCor}
  \frac{M}{2}\leq |\varphi^{(d)}(t)|\leq M,
\end{equation}
for all $t\in I$. Then
\[
\|\E_\g f\|_{L^{p'}(\R^d)}\leq C\|f\|_{L^{q'}(I; \lambda_\g dt)},
\]
holds for all $p$ and $q$ in the range
\[
1\leq p< \frac{d^2+d+2}{d^2+d},\quad q=\frac{2}{d^2+d}p',
\]
and the constant $C$ depends only on $p$ and $d$.
\end{corollary}
\begin{proof}
Consider the scaled simple curve
\[
\g_M(t)=(t,t^2,\ldots, t^{d-1}, M^{-1}\varphi(t)).
\]
Then the assumptions of Lemma \ref{Restriction-Chen} are satisfied for $\g_M$. Moreover, through a change of variables argument we have
\[
M^{\frac{1}{p'}}\|\F_\g f\|_{L^{p'}(\R^d)}=\|\F_{\g_M}f\|_{L^{p'}(\R^d)}\leq C\|f\|_{L^{q'}(I)}.
\]

From \eqref{assumptionsResCor} we have
\[
\left(\frac{M}{2}\right)^\frac{2}{d^2+d}\leq \lambda_\g(t)\leq M^{\frac{2}{d^2+d}},
\]
as $\lambda_\g(t)=|\varphi^{(d)}(t)|^{\frac{2}{d^2+d}}$. In particular, we can bound the weighed extension operator point-wise by the unweighted, since
\[
\E_\g f(x)=\int_I e^{2\pi i x\cdot \g(t)}f(t)\lambda(t)dt\leq M^{\frac{2}{d^2+d}}\int_I e^{2\pi i x\cdot \g(t)}f(t)dt=M^{\frac{2}{d^2+d}}\F_\g f(x).
\]
Therefore, using the fact that $q/p'=2/(d^2+d)$, we have
\[
\|\E_\g f\|_{L^{p'}(\R^d)}\leq M^{\frac{q}{p'}}\|\F_\g f\|_{L^{p'}(\R^d)}=M^{\frac{q-1}{p'}}\|\F_{\g_M}f\|_{L^{p'}(\R^d)}\leq C M^{\frac{q}{q'p'}}\|f\|_{L^{q'}(I)}\leq C\|f\|_{L^{q'}(I;\lambda_\g dt)},
\]
where we used that $q-1=q/q'$ and $(M/2)^{q/p'}\leq \lambda_\g(t)$.
\end{proof}

The following proposition is proven in Section $3$ in \cite{Stovall}, and follows from the theory of Littlewood-Paley square-function estimates, see for instance chapter $7$ of \cite{Muscalu}.
\begin{prop}[Proposition $3.2$ in \cite{Stovall}]\label{SqFuncEst}
Let $\varphi\in \mathcal{A}_0^{d-1}(I,N,R)$ and let $I_j$ be the intervals from Theorem \ref{MainDecompTheorem}. For $n\in\mathbb{Z}$, define
\begin{equation}\label{SubCollectionInt}
    I_{j,n}:=\{t\in I_j: 2^{n-1}<|t-a_j|\leq 2^n\}.
\end{equation}
Then for each $(p,q)$ satisfying $(d^2+d)q=2p'$ and $p'>(d^2+d+2)/2$, and for any $f\in L^p$ and any fixed $j$,
\[
\|\mathcal{E}_\g(\chi_{I_j}f)\|_{L^{p'}(\R^d)}\leq C\left\|\left(\sum_{n\in\mathbb{Z}}|\mathcal{E}_\g(\chi_{I_{j,n}}f)|^2\right)^\frac{1}{2}\right\|_{L^{p'}(\R^d)},
\]
where the constant $C$ depends only on $N$, $d$ and $p$.
\end{prop}

The remaining part of this section is more or less due to Stovall as we will closely follow the ideas of Section $4$ in \cite{Stovall}. For completeness, we have decided to include the full proofs as some
minor changes were needed.
The next lemma is a key component in the proof of Theorem \ref{MainRestrictionTheorem}, and is a modified version of Lemma $4.1$ in \cite{Stovall}.
\begin{lemma}\label{MultiLem}
Let $I$ be an interval. Assume that there exists $C_1,C_2>0$, $\kappa\in\mathbb{N}$ and $a\in\mathbb{R}$ such that $C_1|t-a|^{\kappa}\leq|L_\g(t)|\leq C_2|t-a|^{\kappa}$ for each $t\in I$. Let $n_1\leq\ldots \leq n_D$ be a finite sequence of integers, and assume that for each $1\leq j\leq D$ the function $f_j\in L^{q}(I;\lambda_\g dt)$ is supported on $I_{n_j}:=\{t\in I: 2^{n_j-1}<|t-a|\leq 2^{n_j}\}$. Then there exist constants $K=K_{d}>0$ and $C=C_{d,N,p,q}>0$ such that
\[
\left\|\prod_{l=1}^D\mathcal{E}_\g f_l\right\|_{L^{\frac{p'}{D}}(\R^d)}\leq C2^{-K(n_D-n_1)}\prod_{l=1}^D\|f_l\|_{L^{q'}(I;\lambda_\g dt)}.
\]
\end{lemma}
\begin{proof}
By the generalized H\"{o}lder's inequality it follows for any $\sigma\in S_D$, where $S_D$ denotes the group of permutations on $\{1,\ldots,D\}$, and any integer $1\leq M\leq D-1$, that
\[
\left\|\prod_{l=1}^D\mathcal{E}_\g f_l\right\|_{L^{\frac{p'}{D}}(\R^d)}\leq \prod_{l=N+1}^D\left\|\mathcal{E}_\g f_{\sigma(l)}\right\|_{L^{p'}(\R^d)}\left\|\prod_{j=1}^M\mathcal{E}_\g f_{\sigma(j)}\right\|_{L^{\frac{p'}{M}}(\R^d)}.
\]
Thus, it is enough to show that
\begin{equation}\label{Reduced multi-est}
\left\|\prod_{j=1}^d\mathcal{E}_\g f_{\sigma(j)}\right\|_{L^{\frac{p'}{D}}(\R^d)}\leq C2^{-K\left(n_{\sigma(d)}-n_{\sigma(1)}\right)}\prod_{l=1}^d\|f_l\|_{L^{q'}(I;\lambda_\g dt)},
\end{equation}
holds for any choice of permutation $\sigma\in S_D$ by choosing $M=d$, and for the remaining terms using that
\begin{equation}\label{trivialMulti}
  \left\|\prod_{l=1}^D\mathcal{E}_\g f_l\right\|_{L^{\frac{p'}{D}}(\R^d)}\leq \prod_{l=1}^D\left\|\mathcal{E}_\g f_{\sigma(l)}\right\|_{L^{p'}(\R^d)}\leq C\prod_{l=1}^D\|f_l\|_{L^{q'}(I; \lambda_\g dt)},
\end{equation}
by corollary \ref{Restriction-Chen-Cor} as each $f_j$ is supported on $I_{n_j}$. 

Assume \eqref{Reduced multi-est} holds for the case $q_0=2$, which is also the case $p_0'=d^2+d$ as we are on the scaling line. Then, for any $p'=Dq$, and $\theta\in(0,1)$,
\begin{equation*}
    \frac{1}{q_\theta}=\frac{\theta}{q}+\frac{1-\theta}{2}=D\left(\frac{\theta}{p'}+\frac{1-\theta}{d^2+d}\right)=D\frac{1}{p_\theta'},
\end{equation*}
where $p_\theta$ is given by
\[
\frac{1}{p_\theta}=\frac{\theta}{p}+\frac{1-\theta}{p_0}.
\]
This means that interpolation between points on the scaling line $p'=Dq$ keeps us on the scaling line.
So for any $(p,q)$ satisfying
\[
p<\frac{d^2+d+2}{d^2+d},\quad p'=Dq,
\]
we can find $\varepsilon=\varepsilon(p)>0$ such that $p_\varepsilon=p+\varepsilon<(d^2+d+2)/(d^2+d)$, and $q_\varepsilon=p_\varepsilon'/D$. Applying \eqref{trivialMulti} with $D$ replaced by $d$ for the pair $(p_\varepsilon,q_\varepsilon)$, it follows by Riesz-Thorin's interpolation theorem that
\begin{equation*}
    \left\|\prod_{j=1}^d\mathcal{E}_\g f_{\sigma(j)}\right\|_{L^{\frac{p'}{D}}(\R^d)}\leq C2^{-K\left(n_{\sigma(d)}-n_{\sigma(1)}\right)}\prod_{l=1}^d\|f_l\|_{L^{q'}(I;\lambda_\g dt)}.
\end{equation*}
It is therefore enough to consider the case $q=2$ and $p'=d^2+d$.

Without loss of generality we may assume that $\sigma(j)=j$ in \eqref{Reduced multi-est} for all $1\leq j\leq d$. Two cases will be considered separately. The first case is $n_d-n_1\leq d$. Then, by \eqref{trivialMulti} it follows that
\begin{equation}\label{FirstCase}
    \left\|\prod_{l=1}^D\mathcal{E}_\g f_l\right\|_{L^{\frac{p'}{D}}(\R^d)}\leq  C\prod_{l=1}^D\|f_l\|_{L^{q'}(I; \lambda_\g dt)}\leq C_12^d2^{-(n_d-n_1)}\prod_{l=1}^D\|f_l\|_{L^{q'}(I;\lambda_\g dt)}.
\end{equation}
which is \eqref{Reduced multi-est}.

Assume $n_d-n_1>d$, and define for each $j$ the measure $\mu_j$ ,which acts as a linear functional on $C_0(\R^d)$ through
\[\mu_j(\phi)=\int_{\R^d}\phi d\mu_j=\int_{I_{n_j}}\phi(\g(t))f_j(t)\lambda_\g(t)dt,\]
for each $\phi\in C_0(\R^d)$, and
where $\lambda_\g(t)=|L_\g(t)|^{\frac{2}{d(d+1)}}$.
The measure $\mu_j$ is defined such that $\E_\g f_j=\widehat{\mu}_j$. In particular, Hausdorff-Young's inequality gives
\begin{align}\label{Hausdorff-Young}\left\|\prod_{j=1}^d \mathcal{E}_\g f_{j}\right\|_{L^{d+1}(\R^d)}=\left\|\prod_{j=1}^d \widehat{\mu}_{j}\right\|_{L^{d+1}(\R^d)}
=&\left\|\left({\mu}_{j}*\ldots*\mu_d\right)^\wedge\right\|_{L^{d+1}(\R^d)}\nonumber\\
\leq& \|\mu_i*\ldots*\mu_d\|_{L^{\frac{d+1}{d}}(\R^d)}.\end{align}
Computing the $d$-fold convolution of the measure acting on a fixed $\phi\in C_0(\R)$ results in
\begin{align*}
\mu_1*\ldots*\mu_d(\phi)=&\int_{I_{n_1}}\ldots\int_{I_{n_d}}\phi\left(\sum_{j=1}^d \g(t_j)\right) \prod_{l=1}^d f_l(t_l)\lambda_\g(t_l)dt_l\\
=& \sum_{\sigma\in S_d}\int_{P_\sigma}\phi\left(\sum_{j=1}^d \g(t_j)\right) \prod_{l=1}^d f_l(t_l)\lambda_\g(t_l)dt,
\end{align*}
where $S_d$ is the group of permutations on $\{1,\ldots,d\}$, and where
\[P_\sigma=\left\{(t_1,\ldots,t_d)\in I_{n_1}\times\ldots\times I_{n_d}:t_{\sigma(1)}<\ldots<t_{\sigma(d)}\right\},\]
for each $\sigma\in S_d$. The map $(t_1,\ldots,t_d)\mapsto \sum_{j=1}^d\g(t_j)$ is injective on $P_\sigma$. Denote this map by $\Gamma$, namely\[
\Gamma(t)=\sum_{j=1}^d\g(t_j).
\]
Through a change of variables $y:=\Gamma(t)$ it is possible to write the Radon-Nikodym derivative of the $d$-fold convolution measure as
\begin{equation}\label{Radon-Nikodym-Derivative}\frac{d(\mu_1*\ldots*\mu_d)}{dt}=\sum_{\sigma\in S_d}F_\sigma,\end{equation}
where the functions $F_\sigma$ are given by
\begin{align*}
F_\sigma(y)=&\chi_{P_\sigma}(t)\left(\prod_{l=1}^d f_l(t_l)\lambda_\g(t_l)|J_\Gamma(t)|^{-1}\right)\Big|_{t=\left(\Gamma|_{P_\sigma}\right)^{-1}(y)}\\=&\chi_{P_\sigma}(\Gamma^{-1}(y))\left(\bigotimes_{l=1}^d f_l\lambda_\g\right)(\Gamma^{-1}(y))|J_\Gamma(\Gamma^{-1}(y))|^{-1}\\
=&\left(\chi_{P_\sigma}|J_\Gamma|^{-1}\left(\bigotimes_{l=1}^d f_l\lambda_\g\right)\right)(\Gamma^{-1}(y)).
\end{align*}
Here $f\otimes g(t_1,t_2)=f(t_1)g(t_2)$. Fix $\sigma\in S_d$. Then by reversing the change of variables $y=\Gamma(t)$, one sees that
\begin{align*}
\|F_\sigma\|^{\frac{d+1}{d}}_{L^{\frac{d+1}{d}}(\R^d)}=&\int_{\R^d}|F_\sigma(y)|^{\frac{d+1}{d}}dy\\
=&\int_{\R^d}\left|\left(\chi_{P_\sigma}|J_\Gamma|^{-1}\left(\bigotimes_{l=1}^d f_l\lambda_\g\right)\right)(\Gamma^{-1}(y))\right|^{\frac{d+1}{d}}dy\\
=&\int_{\R^d}\left|\left(\chi_{P_\sigma}\left(\bigotimes_{l=1}^d f_l\lambda_\g\right)\right)(t)\right|^{\frac{d+1}{d}}|J_\Gamma(t)|^{-\frac{1}{d}}dt\\
=&\left\|\left(\chi_{P_\sigma}\left(\bigotimes_{l=1}^d f_l\lambda_\g\right)\right)|J_\Gamma|^{-\frac{1}{d+1}}\right\|^{\frac{d+1}{d}}_{L^{\frac{d+1}{d}}(\R^d)}.
\end{align*}
By the geometric inequality \eqref{GeometricInequality}, and the fact that $\lambda_\g(t)=|\varphi^{(d)}(t)|^\frac{2}{d^2+d}$, it follows that,
\begin{align*}
\|F_\sigma\|^{\frac{d+1}{d}}_{L^{\frac{d+1}{d}}(\R^d)}
=&\int_{P_\sigma}\prod_{l=1}^d |f_l\lambda_\g(t_l)|^{\frac{d+1}{d}}|J_\Gamma(t_1,\ldots,t_d)|^{-\frac{1}{d}}dt_1\ldots dt_d\\
\leq& C\int_{P_\sigma}\prod_{l=1}^d |f_l\lambda_\g(t_l)|^{\frac{d+1}{d}}\prod_{k=1}^d\lambda_\g(t_k)^{-\frac{d+1}{2d}}\prod_{1\leq i<j\leq d}|t_j-t_i|^{-\frac{1}{d}}dt_1\ldots dt_d\\
=& C\int_{P_\sigma}\prod_{l=1}^d |f_l(t_l)|^{\frac{d+1}{d}}|\lambda_\g(t_l)|^{\frac{d+1}{2d}}\prod_{1\leq i<j\leq d}|t_j-t_i|^{-\frac1d}dt_1\ldots dt_d.
\end{align*}

Since there are only $d$ choices of $n_j$, there necessarily has to exist some $1\leq k<d$, such that $n_{k+1}-n_k\geq (n_d-n_1)/d>1$. This implies that $n_{k+1}-n_k\geq 2$ as all $n_j$ are integers. Thus, given $(t_1,\ldots,t_d)\in I_{n_1}\times\ldots\times I_{n_d}$, it follows that
\begin{align*}
\prod_{1\leq i<j\leq d}|t_j-t_i|=&\prod_{i\leq k, k+1\leq j}|t_j-t_i|\prod_{1\leq i<j\leq k}|t_j-t_i|\prod_{k+1\leq i<j\leq d}|t_j-t_i|.
\end{align*}
From the definition of $I_{n_j}$, we have that $|a-t_{j}|> 2^{n_{j}-1}$, while $|a-t_i|\leq 2^{n_i}$. Whence it follows by the reverse triangle inequality that for any $j\geq k+1$ and $i\leq k$,
\begin{equation*}
|t_j-t_i|\geq |t_{j}-a+a-t_i|\geq 2^{n_{j}-1}-2^{n_i}\geq 2^{n_j}(2^{-1}-2^{n_k-n_j})\geq2^{n_j-2},
\end{equation*}
since $n_j\geq n_{k+1}\geq n_k+2\geq n_i+2$. Thus, for each fixed $j\geq k$,
\[\prod_{i\leq k}|t_j-t_i|\geq \prod_{i\leq k}2^{n_k-2}=\frac{2^{kn_j}}{4^k}\geq\frac{2^{kn_j}}{4^d}=C_d2^{kn_j}.\]
Considering all $j\geq k+1$, yields
\[\prod_{i\leq k,k+1\leq j}|t_j-t_i|\geq C_d2^{k\sum_{j=k+1}^d n_j}.\]
Thus, returning to the estimate of the function $F_\sigma$, it can be bounded by
\begin{align}\label{FsigmaBound}
&\|F_\sigma\|^{\frac{d+1}{d}}_{L^{\frac{d+1}{d}}(\R^d)}\nonumber 
\\
&\leq C 2^{-\frac{k}{d}\sum_{j=k+1}^dn_{j}}\int_{P_\sigma}\prod_{l=1}^d |f_l\lambda^{\frac{1}{2}}_\g(t_j)|^{\frac{d+1}{d}}\prod_{1\leq i<j\leq k}|t_j-t_i|^{-\frac1d}\prod_{k+1\leq i<j\leq d}|t_j-t_i|^{-\frac1d}dt_1\ldots dt_d \notag \\
&=C 2^{-\frac{k}{d}\sum_{j=k+1}^dn_{j}}T_k(f_1,\ldots,f_k)T_{d-k}(f_{k+1},\ldots, f_d),
\end{align}
where the map $T_\eta$ is defined by
\begin{equation*}
T_\eta(\psi_1,\ldots,\psi_\eta):=\int_{P_\sigma}\prod_{l=1}^\eta |\psi_l(t_l)|^{\frac{d+1}{d}}|\lambda_\g(t_l)|^{\frac{d+1}{2d}}\prod_{1\leq i<j\leq \eta}|t_j-t_i|^{-\frac1d}dt_1\ldots dt_\eta.
\end{equation*}
By multilinear interpolation, see the proof of Proposition 2.2 in \cite{Christ}, Christ shows that if $1\leq p<l$ and $p^{-1}+q^{-1}(l-1)/2=1$, then there exists a constant $C>0$ such that
\[\int_{\R^\eta}\prod_{i=1}^\eta \psi_i(t_i)\prod_{1\leq i<j\leq \eta}g_{ij}(t_i-t_j)dt_1\ldots dt_\eta\leq C \prod_{i=1}^\eta \|\psi_i\|_{L^p{(\R)}}\prod_{1\leq i<j\leq \eta}\|g_{ij}\|_{L^{q,\infty}(\R)}.\]
If we apply this result to $T_k$ with $q=d$ and $p=2d/(2d-k+1)$ we see that
\[|T_k(f_1,\ldots,f_k)|\leq C \prod_{j=1}^k\|(f_j\lambda_\g^{\frac{1}{2}})^{\frac{d+1}{d}}\|_{L^{\frac{2d}{2d-k+1}}(\R)}=C\prod_{j=1}^k\|f_j\lambda_\g^{\frac{1}{2}}\|_{L^{\frac{2(d+1)}{2d-k+1}}(\R)}^{\frac{d+1}{d}},\]
as the Vandermonde determinant is weakly integrable.
Since $k< d$, it follows that
\[\frac{2d+2}{2d-k+1}<\frac{2(d+1)}{d+1}= 2.\]
As $\text{supp } f_j\subset I_{n_j}$ and $|I_{n_j}|\leq 2^{n_j}$, it follows by Hölder's inequality that for each $j$,
\begin{align*}
\|f_j\lambda_\g^{\frac{1}{2}}\|_{L^{\frac{2(d+1)}{2d-k+1}}(\R)}^{\frac{2d+2}{2d-k+1}}=&\int_\R |f_j(t)\lambda_\g^{\frac{1}{2}}(t)|^{\frac{2d+2}{2d-k+1}}dt\\
\leq& |I_{n_j}|^{\frac{1}{r'}}\left\||f_j(t)\lambda_\g^{\frac{1}{2}}(t)|^{\frac{2d+2}{2d-k+1}}\right\|_{L^r(\R)}\\
=&2^{\frac{n_j(d-k)}{2d-k+1}}\|f_j\lambda_\g^{\frac{1}{2}}\|_{L^2(\R)}^{\frac{2d+2}{2d-k+1}},
\end{align*}
where
\[r=\frac{4d-2k+2}{2d+2}.\]
This is necessarily the same as
\[\|f_j\lambda_\g^{\frac{1}{2}}\|_{L^{\frac{2(d+1)}{2d-k+1}}(\R)}^{\frac{d+1}{d}}
\leq C2^{\frac{n_j(d-k)}{2d}}\|f_j\lambda_\g^{\frac12}\|_{L^2(\R)}^{\frac{d+1}{d}}
=C2^{\frac{ n_j}{d}\frac{d-k}{2}}\|f_j\|_{L^2(I;\lambda_\g dt)}^{\frac{d+1}{d}}.\]
Thus an upper bound on $T_k$ is given by
\begin{equation}\label{TkBound}|T_k(f_1,\ldots,f_k)|\leq C \prod_{j=1}^k2^{\frac{n_j}{d}\frac{d-k}{2}}\|f_j\|_{L^2(I;\lambda_\g dt)}^{\frac{d+1}{d}}.\end{equation}
Replacing $k\mapsto d-k$, a similar bound for $T_{d-k}$ is given by
\begin{equation}\label{TdkBound}|T_{d-k}(f_{k+1},\ldots,f_{d})|\leq C \prod_{j=k+1}^d2^{\frac{n_j}{d}\frac{k}{2}}\|f_j\|_{L^2(I;\lambda_\g dt)}^{\frac{d+1}{d}}.\end{equation}
Combining \eqref{FsigmaBound}, \eqref{TkBound}, \eqref{TdkBound}, and that $kd\geq 2$ results in
\begin{align*}
\|F_\sigma\|_{L^{\frac{d+1}{d}}(\R^d)}^{\frac{d+1}{d}}\leq&C 2^{\frac{1}{2d^2}\left((d^2-kd)\sum_{l=1}^kn_l+(kd-2kd)\sum_{l=k+1}^dn_l\right)}\prod_{j=1}^d\|f_j\|_{L^2(I;\lambda_\g dt)}^{\frac{d+1}{d}}\\
\leq &C 2^{\frac{1}{2d^2}(kd(d-k))(n_k-n_{k+1})}\prod_{j=1}^d\|f_j\|_{L^2(I;\lambda_\g dt)}^{\frac{d+1}{d}}.
\end{align*}
Here the fact that $n_l\leq n_{k}$ for $l\leq k$, and $n_l\geq n_{k+1}$ for $l\geq k+1$ was used. Applying that $n_{k+1}-n_k\geq (n_d-n_1)/d$ yields
\begin{equation}\label{FsigmaDiadicBound}\|F_\sigma\|_{L^{\frac{d+1}{d}}(\R^d)}\leq C 2^{-\frac{k(d-k)}{2d(d+1)}(n_{d}-n_{1})}\prod_{j=1}^d\|f_j\|_{L^2(I;\lambda_\g dt)}.\end{equation}
Combining \eqref{Hausdorff-Young}, \eqref{Radon-Nikodym-Derivative}, \eqref{FsigmaDiadicBound} and the triangle inequality we end up with
\begin{align*}
    \left\|\prod_{j=1}^d \mathcal{E}_\g f_{j}\right\|_{L^{d+1}(\R^d)}
    \leq \left\|\sum_{\sigma \in S_d}F_\sigma\right\|_{L^{\frac{d+1}{d}}(\R^d)}
    \leq C_2 2^{-K_{d,k}(n_{d}-n_{1})}\prod_{j=1}^d\|f_j\|_{L^2(I;\lambda_\g dt)},
\end{align*}
where the constant $K_{d,k}$ is given by
\[K_{d,k}=\frac{k(d-k)}{2d(d+1)}\geq \frac{d-1}{2d(d-1)}=\frac{1}{2d}>0,\]
as the quadratic form $kd-k^2$ obtains its minimum at the boundary, which is given by $k=1$, or $k=d-1$.

The proof is concluded by combining the case $n_d-n_1\leq d$ with the case $n_d-n_1>d$ by choosing
\[C=\max\{2^dC_1,C_2\},\quad 
K=\min\left\{1, K_{d,1},\ldots,K_{d,d-1}\right\}=\frac{1}{2d}.
\]
\end{proof}


\begin{proof}[Proof of Theorem \ref{MainRestrictionTheorem}]
Let 
\[
p'=\frac{d^2+d}{2}q,\quad D=\frac{d^2+d}{2}\in\mathbb{N}.
\]
Note that $p'=Dq$, and assume $p'\leq d^2+d$.
By the finite decomposition of Theorem \ref{MainDecompTheorem} and the triangle inequality, it is sufficient to consider one interval. The interval $I_j$ will remain fixed for the rest of the proof. By Proposition \ref{SqFuncEst}, followed by Minkowski's inequality, we see that
\begin{align*}
    \|\E_\g (\chi_{I_j}f)\|_{L^{p'}(\R^d)}^{p'}\leq& C\left\|\left(\sum_{n\in\mathbb{Z}}|\mathcal{E}_\g(\chi_{I_{j,n}}f)|^2\right)^\frac{1}{2}\right\|_{L^{p'}(\R^d)}^{p'}\\
    =&C\int_{\R^d}\left(\sum_{n\in\mathbb{Z}}|\mathcal{E}_\g(\chi_{I_{j,n}}f)(x)|^2\right)^\frac{p'}{2}dx\\
    =&C\int_{\R^d}\prod_{l=1}^D\left(\sum_{n_l\in\mathbb{Z}}|\mathcal{E}_\g(\chi_{I_{j,n_l}}f)(x)|^2\right)^\frac{p'}{2D}dx\\
    \leq& \int_{\R^d}\prod_{l=1}^D\sum_{n_l\in\mathbb{Z}}|\E_\g (\chi_{I_{j,n_l}}f)(x)|^{\frac{p'}{D}}dx\\
    =&D!\sum_{n_1\leq \ldots \leq n_D}\int_{\R^d}\prod_{l=1}^D|\E_\g (\chi_{I_{j,n_l}}f)(x)|^{\frac{p'}{D}}dx,
\end{align*}
and applying Lemma \ref{MultiLem} we get that
\begin{equation*}
    \|\E_\g (\chi_{I_j}f)\|_{L^{p'}(\R^d)}^{p'}\leq C_{d,N,p,q}\sum_{n_1\leq \ldots \leq n_D}2^{-K(n_D-n_1)}\prod_{l=1}^D\|\chi_{I_{j,n_{l}}}f\|_{L^{q'}(I;\lambda_\g dt)}^\frac{p'}{D}.
\end{equation*}
Define the set \[J_{j,n_1,n_D}=[a_j+2^{n_1-1},a_j+2^{n_D}]\cup [a_j-2^{n_D},a_j-2^{n_1-1}].\] Since $n_1\leq\ldots\leq n_D$, it follows that for any
\[
x\in I_{j,n_l}:=\{t\in I_j: 2^{n_l-1}<|t-a_j|\leq 2^{n_l}\},
\]
we must have $a_j+2^{{n_l-1}}<x\leq a_j+2^{n_l}$ or $a_j-2^{n_l}\leq x<a_j-2^{n_l-1}$. In particular, this implies that
$x\in J_{j,n_1,n_D}$.
Moreover, for fixed $n_1$ and $n_D$ there cannot exist more than $(n_D-n_1+1)^{D-2}$
different ways to assign the integers $n_2,\ldots n_{D-1}$ such that $n_1\leq n_2\leq\ldots \leq n_{D-1}\leq n_{D}$. 
As such, let $m=n_D-n_1+1$, so that
\begin{equation}\label{DobbelSumEst}
    \|\E_\g (\chi_{I_j}f)\|_{L^{p'}(\R^d)}^{p'}\leq C\sum_{m=1}^\infty\sum_{n\in\mathbb{Z}}2^{-Km}m^{D-2}\|\chi_{J_{j,n,n+m}}f\|_{L^{q'}(I;\lambda_\g dt)}^{p'}.
\end{equation}
Fix $m$ and assume $l\leq n$. Then by comparing the endpoints, we see that the two sets $J_{j,l,l+m}$ and $J_{j,n,n+m}$ have a non-empty intersection if and only if,
\begin{equation*}
    2^{n-1}\leq 2^{l+m},
\end{equation*}
which means that $n-l\leq m+1$. Thus, for each point $x\in I$, there are at most $m+1$ intervals $I_{j,n,n+m}$ which contain $x$, and so
\begin{equation*}
    \sum_{n\in\mathbb{Z}}|\chi_{J_{j,n,n+m}}f(x)|^{q'}\leq (m+1)|f(x)|^{q'}.
\end{equation*}
Using the trivial bound, one can estimate the sum over $n$ by,
\begin{align*}
    \sum_{n\in\mathbb{Z}}\|\chi_{J_{j,n,n+m}}f\|_{L^{q'}(I;\lambda_\g dt)}^{p'}
    \leq& \left(\sup_{n\in\mathbb{Z}}\|\chi_{J_{j,n,n+m}}f\|_{L^p(I;\lambda_\g dt)}\right)^{p'-q'}\sum_{n\in\mathbb{Z}}\|\chi_{J_{j,n,n+m}}f\|_{L^{q'}(I;\lambda_\g dt)}^{q'}\\
    \leq& \|f\|_{L^{q'}(I;\lambda_\g dt)}^{q'-p'}\sum_{n\in\mathbb{Z}}\|\chi_{J_{j,n,n+m}}f\|_{L^{q'}(I;\lambda_\g dt)}^{q'}\\
    =&\|f\|_{L^p(\lambda_\g)}^{q'-p'}\sum_{n\in\mathbb{Z}}\int_{I}|\chi_{J_{j,n,n+m}}f(t)|^{q'}\lambda_\g(t)dt\\
    \leq& (m+1)\|f\|^{p'}_{L^{q'}(I;\lambda_\g dt)}.
\end{align*}
Combined with \eqref{DobbelSumEst} this yields
\[
\|\E_\g (\chi_{I_j}f)\|_{L^{p'}(\R^d)}^{p'}\leq C\sum_{m=1}^\infty2^{-Km}(m+1)^{D-1}\|f\|_{L^{q'}(I;\lambda_\g dt)}^{p'}\leq C_{d,N,p,q}\|f\|_{L^{q'}(I;\lambda_\g dt)}^{p'},
\]
which completes the proof.
\end{proof}

\printbibliography
\end{document}